\documentclass[a4paper,12pt]{amsart}

\usepackage{etex}
\usepackage{pgf}
\usepackage{amscd,amsmath,amssymb,amstext,amsfonts,amsthm,latexsym}
\usepackage{verbatim, url}
\usepackage[english]{babel}
\usepackage[latin1]{inputenc} 
\usepackage{graphicx}
\usepackage{mathrsfs,ulem}

\usepackage[left=2cm,right=2cm, top=3.5cm, bottom=3.5cm]{geometry}

\usepackage[all]{xy}
\usepackage{longtable}
\usepackage{fancyhdr}
\usepackage{setspace}
 \usepackage[usenames,dvipsnames]{pstricks}
 \usepackage{epsfig}
\usepackage{caption}
\usepackage{hyperref}

\newtheorem{theorem}{Theorem}[section]
  
\newtheorem{lemma}[theorem]{Lemma}
\newtheorem{corollary}[theorem]{Corollary}
\newtheorem{proposition}[theorem]{Proposition}

\theoremstyle{remark}
\newtheorem{remark}[theorem]{Remark}
\newtheorem{comprem}[theorem]{Computational Remark}

\theoremstyle{definition}
\newtheorem{definition}[theorem]{Definition}

\newtheorem{step}{Step}
\newtheorem*{Input}{Input}
\newtheorem*{Output}{Output}

\newcommand{\OH}{\mathcal O}

\newcommand{\Aut}{\mathrm{Aut} }

\title{On threefolds isogenous to a product of curves}
\makeatletter

\@addtoreset{equation}{section}
\makeatother
\author[D. FRAPPORTI, C. GLEI\ss NER]{DAVIDE FRAPPORTI, CHRISTIAN GLEI\ss NER}
\address{Davide Frapporti, Christian Glei\ss ner;
 University of Bayreuth, Lehrstuhl Mathematik VIII;
Universit\"atsstra\ss e 30, D-95447 Bayreuth, Germany}

\thanks{ The authors thank I. Bauer, S. Weigl and H. Mohseni Ahouei for several  
suggestions, useful discussions and careful reading of the paper. 	
The present work took mainly place in the realm of the DFG
Forschergruppe 790 ``Classification of algebraic
surfaces and compact complex manifolds''.
 The first author is a member of G.N.S.A.G.A. of I.N.d.A.M.	}

\begin{document}

\begin{abstract} 
A \textit{threefold isogenous to a product of curves} $X$ is a quotient of a product of three compact Riemann 
surfaces of genus at least two by the free action of a finite group.
In this paper we study these threefolds under the assumption that the group acts 
diagonally on the product. We show that the classification of these threefolds is a finite problem,
 present an algorithm to classify them for a fixed value of 
$\chi(\mathcal O_X)$ and explain a  method to determine their Hodge numbers.
Running an implementation of the algorithm 
we achieve the full classification of threefolds isogenous to a product of curves with $\chi(\mathcal O_X)=-1$, 
under the assumption that  the group acts faithfully on each factor.
\end{abstract}

\maketitle

\section*{Introduction}

A complex algebraic variety  $X$ is  \textit{isogenous to a product of curves} if $X$ is a quotient 
\[
X=(C_1 \times \ldots\times C_n)/G,
\]
where the $C_i$'s  are compact Riemann surfaces of genus
at least two and $G$ is a finite group acting freely on $C_1 \times \ldots \times  C_n$. 
This class of  varieties of general type has been
introduced by Catanese in \cite{Cat00}. Since then, a considerable amount of literature appeared, 
especially in the case of surfaces. 
In particular surfaces $S$ isogenous to a product
of curves with $\chi(\OH_S)=1$ (equivalently $p_g(S)=q(S)$) are completely classified,
see \cite{BCG08, CP09,penegini,CCML98, Pir02, HP02, Be82}.

A natural question is:
 ``\,Is it possible to classify varieties $X$ isogenous to a product  for a fixed value of $\chi(\OH_X)$ in higher dimensions?\,"

Let $G^0$ be the diagonal subgroup
\[
G^0:= G \cap \big(\Aut(C_1)\times \ldots \times \Aut(C_n)\big).
\]
There are two possibilities for the action of $G$ on the product $C_1 \times \ldots \times C_n$:
\begin{itemize}
\item  \textit{Unmixed}: $G = G^0$, i.e. the group $G$ acts diagonally. 
\item \textit{Mixed}:  $G \neq G^0$, i.e. there are elements in $G$ permuting some factors of the product.
\end{itemize}
 
In the present article we concentrate on the  unmixed case in dimension three, while in a forthcoming one we plan to 
investigate the mixed case.

Following the approach used by the above mentioned authors,  
we show that the  classification problem can be translated into a problem of group theory.

Let $X:= \big(C_1 \times C_2 \times C_3\big)/G$ be a threefold isogenous to a product 
of unmixed type.
Restricting  the $G$-action  we obtain homomorphisms
\[
\psi_i: G \to \Aut(C_i) \quad \makebox{with kernels} \quad K_i:=\ker(\psi_i).
\]
Note that in contrast to the surface case, it is not possible to assume that the kernels $K_i$ are trivial.

To the threefold $X$ we attach the \textit{algebraic datum}
\[
(G,K_1,K_2,K_3,V_1,V_2,V_3)
\]
where  $V_i$ is a \textit{generating vector} for the group $G/K_i$ (see Definition \ref{DefGV}).

Conversely, by Riemann's existence Theorem, 
a  tuple consisting of a finite group $G$, three normal subgroups $K_i$ and generating vectors $V_i$ for $G/K_i$
 satisfying certain relations, determines a family of threefold isogenous to a product.

What turns the classification into a finite problem is the fact that 
the freeness assumption  for  the group action allows us to 
bound the group order in terms of $\chi(\mathcal O_X)$  
and derive  strong constraints on the orders of the kernels $K_i$, the genera 
$g(C_i)$ and the types $T_i$ of the generating vectors $V_i$. 

This observation is used to develop an algorithm searching systematically through all admissible algebraic data  
and thereby classifying all threefolds
isogenous to a product of unmixed type with a fixed value of $\chi(\mathcal O_X)$. 
Unfortunately the bound of the group order that we have is not good enough to perform 
the complete classification even for small values of $\chi(\mathcal O_X)$ and  using a computer algebra system. 
However, if we restrict to the \textit{absolutely faithful case}, that is all kernels $K_i$ are trivial, 
 then the bound drops significantly making a full classification possible. 
 
For this purpose we implemented our algorithm using the computer algebra system MAGMA \cite{magma}. 
A commented version can be downloaded from
\[
\mbox{\url{http://www.staff.uni-bayreuth.de/~bt300503/}}.
\]

Here we point out a difference in our method and the method of the authors which classified 
surfaces isogenous to a product with $\chi(\mathcal O_S)=1$. 
They considered separately each possible value 
of $p_g(S)=q(S)$, instead we fix only $\chi(\mathcal O_X)$.

This leads us naturally to the question which Hodge numbers are realized by our construction.
In the case of surfaces isogenous to a product, 
once $p_g(S)$ and $q(S)$ are fixed, the remaining Hodge number $h^{1,1}(S)$ is 
uniquely determined. 
In dimension three (and higher), it is not clear how to determine the Hodge diamond, 
even if we would fix the invariants
\[
p_g(X):= h^3(\OH_X), \quad q_2(X):= h^2(\OH_X) \quad \makebox{and} \quad q_1(X):= h^1(\OH_X).  
\]
Our strategy relies on representation theory. The group actions $\psi_i$ induce representations
via pull-back of holomorphic $1$-forms 
\[
\varphi_i: G \to GL\left( H^{1,0}(C_i) \right).
\]
According to the formula of \textit{Chevalley-Weil}, the characters $\chi_{\varphi_i}$ of these representations 
can be computed from a generating vector $V_i$ of $G/K_i$. 
This provides a way to compute the Hodge numbers of $X$:
we   determine the characters $\chi_{p,q}$ of the representations 
\[
 \phi_{p,q} : G \to \mathrm{GL}\big(H^{p,q} (C_1 \times C_2 \times C_3)\big)  
\quad g \mapsto  [\,\omega \mapsto (g^{-1})^{\ast}\omega ]
\]
in terms of the characters $\chi_{\varphi_i}$, then $h^{p,q}(X)$ is the multiplicity of the trivial character of $G$ in $\chi_{p,q}$.

As an application, we run our program in the case $\chi(\OH_X)=-1$. Note  that  this is the boundary value: for a threefold $X$ of general 
type with ample canonical class   it holds 
\[0<c_1(K_X)c_2(X)=-24\chi(\mathcal O_X) \,,\]
according to \cite[Theorem 1.1]{MiyaokaChern} and Riemann-Roch.

 The main result of the paper is the following theorem.

 \begin{theorem}\label{thm51}
Let $X:=(C_1 \times C_2 \times C_3)/G$ be a threefold isogenous to a product of curves of unmixed type. 
Assume that the action of $G$ is absolutely faithful and $\chi(\mathcal O_X)=-1$. 

Then, the tuple
\[
[G,T_1,T_2,T_3, p_g(X), q_2(X), q_1(X), h^{1,1}(X), h^{1,2}(X), H_1(X,\mathbb Z)]
\]
appears in Table \ref{tab1}.

 Conversely, each row is realized by at least one family of threefold 
isogenous to a product of curves of unmixed type with $\chi(\mathcal O_X)=-1$
and absolutely faithful $G$-action. 

All these threefolds have invariants $K_X^3= 48$ and $e(X)=-8$.
\end{theorem}

\begin{longtable}{p{1.6cm}p{1.5cm}p{1.5cm}p{1.5cm}p{1.5cm}p{0.5cm}p{0.5cm}p{0.5cm}p{0.5cm}p{1.0cm}p{2.5cm}}\label{tab1}
 $G$ & $Id$ &$T_1$ & $T_2$ & $T_3$  & $p_g$ & $q_2$ & $q_1$ &$h^{1,1}$ & $h^{1,2}$ & $H_1(X,\mathbb Z)$ \\ \hline \endfirsthead
$G$ & $Id$ &$T_1$ & $T_2$ & $T_3$  & $p_g$ & $q_2$ & $q_1$ &$h^{1,1}$ & $h^{1,2}$ & $H_1(X,\mathbb Z)$\\ \hline \endhead

$\mathfrak A_5$ & $\langle 60,5 \rangle$& $[0; 2, 5^2 ]$ & $[0; 3^2, 5 ]$ & $[0; 2^3, 3 ]$ & $2$ & $0$ & $0$ &$3$ & $6$& $\mathbb Z_2\times \mathbb Z_{30}$\\
$\mathfrak S_4 \times \mathbb Z_2 $ & $\langle 48,48 \rangle$& $[0; 2, 4, 6 ]$ & $[0; 2, 4, 6 ]$ & $[0; 2^5 ]$ & $3$ & $1$ & $0$ &$5$ & $9$ &$\mathbb Z_2^3\times \mathbb Z_4$\\
$ \mathrm{GL}(2,\mathbb F_3)$ & $\langle 48,29 \rangle$ & $[0; 2, 3, 8 ]$ & $[0; 2, 3, 8 ]$ & $[2;-]$ & $5$ & $5$ & $2$ &$11$ & $17$&$\mathbb Z^4$\\
$ \mathrm{GL}(2,\mathbb F_3)$ & $\langle 48,29 \rangle$ & $[0; 2, 3, 8 ]$ & $[0; 2, 3, 8 ]$ & $[2;-]$ & $4$ & $4$ & $2$ &$13$ & $18$& $\mathbb Z^4$\\
$\mathbb Z_3 \rtimes_\varphi \mathcal D_4$ & $\langle 24,8 \rangle$ & $[0; 2, 4, 6 ]$ & $[0; 2, 4, 6 ]$ & $[2;-]$ & $5$ & $5$ & $2$ &$11$ &$17$ & $\mathbb Z^4$\\
$\mathbb Z_3 \rtimes_\varphi \mathcal D_4$  & $\langle 24,8 \rangle$ & $[0; 2, 4, 6 ]$ & $[0; 2, 4, 6 ]$ & $[2;-]$ & $4$ & $4$ & $2$ &$13$ & $18$& $\mathbb Z^4$\\
$ \mathrm{SL}(2,\mathbb F_3)$& $\langle 24,3 \rangle$ & $[0; 3, 3, 4 ]$ & $[0; 3, 3, 4 ]$ & $[2;-]$ & $5$ & $5$ & $2$ &$13$ & $19$& $\mathbb Z^4$\\
$\mathfrak S_4$ & $\langle 24,12 \rangle$ & $[0; 3, 4^2 ]$ & $[0; 2^3, 4 ]$  & $[0; 2^2,3^2]$ & $3$ & $1$ & $0$ &$5$ & $9$&$\mathbb Z_2^2\times \mathbb Z_{24}$\\
$\mathbb Z_8 \rtimes_ \varphi \mathbb Z_2 $ & $\langle 16,8 \rangle$ & $[0; 2, 4, 8 ]$ & $[0; 2, 4, 8 ]$ & $[2; -]$ & $5$ & $5$ & $2$ & $11$ & $17$ &$\mathbb Z_2^2\times \mathbb Z^4$\\
$\mathbb Z_8 \rtimes_ \varphi \mathbb Z_2 $ & $\langle 16,8 \rangle$ & $[0; 2, 4, 8 ]$ & $[0; 2, 4, 8 ]$ & $[2; -]$ & $4$ & $4$ & $2$ & $13$ & $18$&$\mathbb Z_2^2\times \mathbb Z^4$\\
$\mathcal D_4 \times \mathbb Z_2 $ & $\langle 16,11 \rangle$ & $[0; 2^3, 4 ]$ & $[0; 2^3, 4 ]$ & $[0; 2^5]$ & $4$ & $2$ & $0$ & $7$ & $12$& $\mathbb Z_2^3\times \mathbb Z_4^2$\\
$\mathcal D_4 \times \mathbb Z_2 $ & $\langle 16,11 \rangle$ & $[0; 2^3, 4 ]$ & $[0; 2^3, 4 ]$ & $[0; 2^5]$ & $3$ & $1$ & $0$ & $5$ & $9$& $\mathbb Z_2^3\times \mathbb Z_4^2$\\
$\mathcal D_6$ & $\langle 12,4 \rangle$ & $[0; 2^3, 3 ]$ & $[0; 2^3, 6 ]$ & $[1; 2^2]$ & $4$ & $3$ & $1$ & $9$ & $14$& $\mathbb Z_2\times \mathbb Z_4\times \mathbb Z^2$\\
$\mathcal D_6$  & $\langle 12,4 \rangle$ & $[0; 2^3, 3 ]$ & $[0; 2^3, 3 ]$ & $[2;-]$ & $5$ & $5$ & $2$ & $13$ & $19$& $\mathbb Z^4$\\
$\mathcal D_6$ & $\langle 12,4 \rangle$& $[0; 2^3, 3 ]$ & $[0; 2^5 ]$ & $[1;3]$ & $4$ & $3$ & $1$ & $9$ & $14$&$\mathbb Z_2^2\times \mathbb Z^2$\\
$\mathbb Z_3 \times \mathbb Z_2^2$ & $\langle 12,5 \rangle$& $[0; 2, 6^2 ]$ & $[0; 2, 6^2 ]$ & $[2;-]$ & $6$ & $6$ & $2$ & $11$ & $18$& $\mathbb Z^4$\\
$\mathbb Z_3 \times \mathbb Z_2^2$ & $\langle 12,5 \rangle$& $[0; 2, 6^2 ]$ & $[0; 2, 6^2 ]$ & $[2;-]$ & $5$ & $5$ & $2$ & $11$ & $17$& $\mathbb Z^4$\\
$\mathbb Z_3 \times \mathbb Z_2^2$ & $\langle 12,5 \rangle$& $[0; 2, 6^2 ]$ & $[0; 2, 6^2 ]$ & $[2;-]$ & $4$ & $4$ & $2$ & $13$ & $18$& $\mathbb Z^4$\\
$\mathbb Z_3 \times \mathbb Z_2^2$ & $\langle 12,5 \rangle$& $[0; 2, 6^2 ]$ & $[0; 2, 6^2 ]$ & $[2;-]$ & $4$ & $4$ & $2$ & $15$ & $20$& $\mathbb Z^4$\\
$\mathbb Z_3 \rtimes_ \varphi \mathbb Z_4$ & $\langle 12,1 \rangle$& $[0; 3, 4^2 ]$ & $[0; 3, 4^2 ]$ & $[2;-]$ & $5$ & $5$ & $2$ & $13$ & $19$& $\mathbb Z^4$\\
$\mathbb Z_{10}$ & $\langle 10,2 \rangle$& $[0; 2, 5, 10 ]$ & $[0; 2, 5, 10 ]$ & $[2;-]$ & $5$ & $5$ & $2$ & $13$ & $19$& 
$\mathbb Z^4$\\
$\mathbb Z_{10}$ & $\langle 10,2 \rangle$& $[0; 2, 5, 10 ]$ & $[0; 2, 5, 10 ]$ & $[2;-]$ & $6$ & $6$ & $2$ & $11$ & $18$& 
$\mathbb Z^4$\\
$\mathbb Z_{10}$ & $\langle 10,2 \rangle$& $[0; 2, 5, 10 ]$ & $[0; 2, 5, 10 ]$ & $[2;-]$ & $4$ & $4$ & $2$ & $15$ & $20$& $\mathbb Z^4$\\
$\mathcal D_4$ & $\langle 8,3 \rangle$& $[0; 2^3, 4 ]$ & $[1; 2 ]$ & $[0;2^6]$ & $4$ & $3$ & $1$ & $9$ & $14$& $\mathbb Z_2^4\times \mathbb Z^2$\\
$\mathcal D_4$ & $\langle 8,3 \rangle$& $[0; 2^3, 4 ]$ & $[1; 2 ]$ & $[1;2^2]$ & $4$ & $4$ & $2$ & $11$ & $16$&  $\mathbb Z_2\times \mathbb Z^4$\\
$\mathcal D_4$ & $\langle 8,3 \rangle$& $[0; 2^3, 4 ]$ & $[0; 2^2, 4^2 ]$ & $[1;2^2]$ & $4$ & $3$ & $1$ & $9$ & $14$&  $\mathbb Z_2^2\times\mathbb Z_8 \times\mathbb Z^2$\\
$\mathcal D_4$ & $\langle 8,3 \rangle$& $[0; 2^3, 4 ]$ & $[0; 2^3, 4 ]$ & $[2;-]$ & $5$ & $5$ & $2$ & $13$ & $19$& $\mathbb Z_2^2\times\mathbb Z^4$\\
$\mathbb Z_8$ & $\langle 8,1 \rangle$& $[0; 2, 8^2 ]$ & $[0; 2, 8^2 ]$ & $[2;-]$ & $6$ & $6$ & $2$ & $11$ & $18$& $\mathbb Z_2^2\times \mathbb Z^4$\\
$\mathbb Z_8$ & $\langle 8,1 \rangle$& $[0; 2, 8^2 ]$ & $[0; 2, 8^2 ]$ & $[2;-]$ & $4$ & $4$ & $2$ & $15$ & $20$& $\mathbb Z_2^2\times \mathbb Z^4$\\
$Q$ & $\langle 8,4 \rangle$& $[0; 4^3 ]$ & $[0; 4^3 ]$ & $[2;-]$ & $5$ & $5$ & $2$ & $13$ & $19$&  $\mathbb Z_2^4\times \mathbb Z^4$\\
$\mathbb Z_2^3$ & $\langle 8,5 \rangle$& $[0; 2^5 ]$ & $[0; 2^5 ]$ & $[0;2^5]$ & $5$ & $3$ & $0$ & $9$ & $15$& $\mathbb Z_2^3\times \mathbb Z_4^3$\\
$\mathbb Z_2^3$ & $\langle 8,5 \rangle$& $[0; 2^5 ]$ & $[0; 2^5 ]$ & $[0;2^5]$ & $4$ & $2$ & $0$ & $7$ & $12$&$\mathbb Z_2^3\times \mathbb Z_4^3$\\
$\mathbb Z_2^3$ & $\langle 8,5 \rangle$& $[0; 2^5 ]$ & $[0; 2^5 ]$ & $[0;2^5]$ & $4$ & $2$ & $0$ & $7$ & $12$&$\mathbb Z_2^5\times \mathbb Z_4^2$\\
$\mathfrak S_3$ & $\langle 6,1 \rangle$& $[0; 2^2, 3^2 ]$ & $[1; 3 ]$ & $[0;2^6]$ & $4$ & $3$ & $1$ & $9$ & $14$& $\mathbb Z_2^3\times \mathbb Z_6\times \mathbb Z^2$\\
$\mathfrak S_3$ & $\langle 6,1 \rangle$& $[0; 2^2, 3^2 ]$ & $[1; 3 ]$ & $[1;2^2]$ & $4$ & $4$ & $2$ & $11$ & $16$& $\mathbb Z_3\times \mathbb Z^4$\\
$\mathfrak S_3$ & $\langle 6,1 \rangle$& $[0; 2^2, 3^2 ]$ & $[0; 2^2, 3^2 ]$ & $[2;-]$ & $5$ & $5$ & $2$ & $13$ & $19$&$\mathbb Z_3^2\times \mathbb Z^4$\\
$\mathbb Z_6$ & $\langle 6,2 \rangle$& $[0; 2^2, 3^2 ]$ & $[0; 2^2, 3^2 ]$ & $[2;-]$ & $6$ & $6$ & $2$ & $15$ & $22$&$ \mathbb Z^4$\\
$\mathbb Z_6$ & $\langle 6,2 \rangle$& $[0; 2^2, 3^2 ]$ & $[0; 3, 6^2 ]$ & $[2;-]$ & $5$ & $5$ & $2$ & $13$ & $19$&$\mathbb Z_3\times \mathbb Z^4$\\
$\mathbb Z_6$ & $\langle 6,2 \rangle$& $[0; 3, 6^2 ]$ & $[0; 3, 6^2 ]$ & $[2;-]$ & $6$ & $6$ & $2$ & $11$ & $18$&$\mathbb Z_3^2\times \mathbb Z^4$\\
$\mathbb Z_6$ & $\langle 6,2 \rangle$& $[0; 3, 6^2 ]$ & $[0; 3, 6^2 ]$ & $[2;-]$ & $4$ & $4$ & $2$ & $15$ & $20$&$\mathbb Z_3^2\times \mathbb Z^4$\\
$\mathbb Z_5$ & $\langle 5,1 \rangle$& $[0; 5^3 ]$ & $[0; 5^3 ]$ & $[2;-]$ & $6$ & $6$ & $2$ & $11$ & $18$&$\mathbb Z_5^2\times \mathbb Z^4$\\
$\mathbb Z_5$ & $\langle 5,1 \rangle$& $[0; 5^3 ]$ & $[0; 5^3 ]$ & $[2;-]$ & $5$ & $5$ & $2$ & $13$ & $19$&$\mathbb Z_5^2\times \mathbb Z^4$\\
$\mathbb Z_5$ & $\langle 5,1 \rangle$& $[0; 5^3 ]$ & $[0; 5^3 ]$ & $[2;-]$ & $4$ & $4$ & $2$ & $15$ & $20$&$\mathbb Z_5^2\times \mathbb Z^4$\\
$\mathbb Z_2^2$ & $\langle 4,2 \rangle$& $[0; 2^5 ]$ & $[0; 2^6 ]$ & $[1;2^2]$ & $6$ & $5$ & $1$ & $13$ & $20$&$\mathbb Z_2^5\times \mathbb Z_4\times \mathbb Z^2$\\
$\mathbb Z_2^2$ & $\langle 4,2 \rangle$& $[0; 2^5 ]$ & $[0; 2^6 ]$ & $[1;2^2]$ & $5$ & $4$ & $1$ & $11$ & $17$&$\mathbb Z_2^5\times \mathbb Z_4\times \mathbb Z^2$\\
$\mathbb Z_2^2$ & $\langle 4,2 \rangle$& $[0; 2^5 ]$ & $[1; 2^2 ]$ & $[1;2^2]$ & $5$ & $5$ & $2$ & $13$ & $19$&$\mathbb Z_2^3\times \mathbb Z^4$\\
$\mathbb Z_2^2$ & $\langle 4,2 \rangle$& $[0; 2^5 ]$ & $[1; 2^2 ]$ & $[1;2^2]$ & $4$ & $4$ & $2$ & $11$ & $16$&$\mathbb Z_2^3\times \mathbb Z^4$\\
$\mathbb Z_2^2$ & $\langle 4,2 \rangle$& $[0; 2^5 ]$ & $[0; 2^5 ]$ & $[2;-]$ & $6$ & $6$ & $2$ & $15$ & $22$&$\mathbb Z_2^4\times \mathbb Z^4$ \\
$\mathbb Z_2^2$ & $\langle 4,2 \rangle$& $[0; 2^5 ]$ & $[0; 2^5 ]$ & $[2;-]$ & $5$ & $5$ & $2$ & $13$ & $19$&$\mathbb Z_2^4\times \mathbb Z^4$ \\
$\mathbb Z_4$ & $\langle 4,1 \rangle$& $[0; 2^2, 4^2 ]$ & $[0; 2^2, 4^2 ]$ & $[2;-]$ & $6$ & $6$ & $2$ & $15$ & $22$&$\mathbb Z_2^4\times \mathbb Z^4$\\
$\mathbb Z_3$ & $\langle 3,1 \rangle$& $[0; 3^4 ]$ & $[0; 3^4 ]$ & $[2;-]$ & $6$ & $6$ & $2$ & $15$ & $22$ &$\mathbb Z_3^4\times \mathbb Z^4$\\
$\mathbb Z_2$ & $\langle 2,1 \rangle$& $[1; 2^2 ]$ & $[1; 2^2 ]$ & $[2;-]$ & $6$ & $8$ & $4$ & $19$ & $26$& $ \mathbb Z^{8}$\\
$\mathbb Z_2$ & $\langle 2,1 \rangle$& $[0; 2^6 ]$ & $[1; 2^2 ]$ & $[2;-]$ & $6$ & $7$ & $3$ & $17$ & $24$& $ \mathbb Z_2^4\times\mathbb Z^6$\\
$\mathbb Z_2$ & $\langle 2,1 \rangle$& $[0; 2^6 ]$ & $[0; 2^6 ]$ & $[2;-]$ & $8$ & $8$ & $2$ & $19$ & $28$& $ \mathbb Z_2^8\times\mathbb Z^4$\\
$\{1\}$ & $\langle 1,1 \rangle$& $[2; - ]$ & $[2; - ]$ & $[2;-]$ & $8$ & $12$ & $6$ & $27$ & $36$& $ \mathbb Z^{12}$		
\\
\\
\caption[Threefolds isogenous]{Threefolds isogenous to a product with $\chi(\mathcal O_X)=-1$.} 
\end{longtable}

In Table \ref{tab1} the types $T_i$ of the generating vectors (see Definition \ref{DefGV}) are given in a simplified way:
\[
[g';a_1^{k_1},\ldots,a_r^{k_r}]:=
[g';\underbrace {a_1,\ldots,a_1}_{\makebox{$k_1$-times}},\ldots,\underbrace {a_r,\ldots,a_r}_{\makebox{$k_r$-times}}].
\]
The column $Id$ reports the MAGMA identifier of the group $G$: here $\langle a,b \rangle$ denotes the $b^{th}$ 
group of order $a$ in the database of Small Groups. The cyclic group or order $n$ is denoted by $\mathbb Z _n$, 
the symmetric group on $n$ letters by $\mathfrak S _n$, the alternating group on $n$ letters by  $\mathfrak A_n$, the 
quaternion group by  $Q$ and the dihedral group of order $2n$ by $\mathcal D_n$.
 
Throughout the paper all varieties are defined over the field of complex numbers
and the standard notation from complex algebraic geometry is used,
see for example \cite{G-H}.

The paper is organized in the following way.

In Section \ref{generalities} we introduce varieties isogenous to a product of curves
and explain some of their basic properties.

Section \ref{Riemannsurf} is about group actions on compact Riemann surfaces: we
recall \textit{Riemann's existence theorem} and state the \textit{Chevalley-Weil formula}
which describes the induced action on the space of holomorphic $1$-forms. 

  In Section \ref{group_descr}
we show that the geometry of a threefold isogenous  to a product of unmixed type is 
encoded in the group via an algebraic datum.
Based on that, we give a method to compute its \textit{Hodge diamond} and its \textit{fundamental group}.

 In Section \ref{bounds_smooth}  we show that the classification of threefolds isogenous to a 
product $X$ of unmixed type with a fixed value of $\chi(\mathcal{O}_X)$  is a finite problem and give an algorithm
to perform this classification.

\newpage
\section{Basics on Varieties Isogenous to a Product}\label{generalities}

\begin{definition}
A complex algebraic variety $X$ is \textit{isogenous to a  product of curves}
if there exist compact Riemann surfaces $C_1, \ldots , C_n$ of genus at least two and 
a finite group  $G \leq \Aut(C_1 \times \ldots \times C_n) $  acting freely on the product $C_1 \times \ldots \times C_n$ such that
\[
X=(C_1 \times \ldots \times C_n)/G.
\]
\end{definition}

It follows from the definition that a variety  $X$ which is isogenous to a product is smooth, projective, 
of general type (i.e.~$\kappa(X)=\dim(X)=n$) and  its  canonical class $K_X$ is ample.

The \textit{$n$-fold self-intersection} of the canonical class $K_X^n$, the \textit{topological Euler number} 
$e(X)$ and the holomorphic \textit{Euler-Poincar\'e-characteristic} $\chi(\mathcal O_X)$ 
can be expressed in terms of the genera $g(C_i)$ and the group order $|G|$.

\begin{proposition}\label{invarsmooth}
Let $X := (C_1 \times \ldots \times C_n)/G $ be a variety isogenous to a product.
Then 
\[
K_X^n= \frac{n! \, 2^n}{|G|} \prod_{i=1}^{n}\big( g(C_i) - 1 \big), \quad 
e(X)= \frac{(-1)^n \, 2^n}{|G|} \prod_{i=1}^{n}\big( g(C_i) - 1 \big)
\]
and
\[
\chi(\mathcal O_X)= \frac{(-1)^n}{|G|} \prod_{i=1}^n \big( g(C_i)-1 \big).
\]
\end{proposition}

\begin{proof}
We define $Y:=C_1 \times \ldots \times C_n$ and denote by $F_i$ a fiber of 
the natural projection $p_i \colon Y \to C_i$ for $1\leq i\leq n$. It holds 
\[   
K_Y \equiv_{num} \sum_{i=1}^n (2g(C_i)-2)F_i \quad \mbox{ hence }\quad 
K_Y^n= n!\, 2^n\prod_{i=1}^{n}(g(C_i) - 1)\,.  
\]

\noindent  The product properties of $e$ and $\chi$ imply
\[
e(Y)=\prod_{i=1}^{n}\big(2-2g(C_i)\big)\  \makebox{ and } \  \chi(\mathcal O_Y)=   \prod_{i=1}^n \big(1- g(C_i)\big)\,.
\] 

\noindent  The statement follows since the quotient map $\pi \colon Y \to X$ is an unramified covering of degree $\deg(\pi)= |G|$.
\end{proof}

\begin{definition}
Let $X=\big(C_1 \times \ldots  \times C_n \big)/G$ be a variety isogenous to a product.  The 
diagonal subgroup of $G$  is defined as 
\[
G^0:= G \cap \big(\Aut(C_1)\times \ldots\times \Aut(C_n)\big).
\]
If $G^0=G$, the action of $G$ is called \textit{unmixed} and \textit{mixed} otherwise.  
\end{definition}

 In the unmixed case the action of $G$ can be restricted to each curve $C_i$. In this way we obtain homomorphisms 
\[
\psi_i\colon G \to \Aut(C_i) \quad \makebox{with kernels} \quad K_i:=\ker(\psi_i). 
\]

In the rest of the paper we shall consider only the case $n=3$ and assume that the action of $G$ is unmixed.
To simplify the notation we write 
\textit{threefold isogenous to a product} keeping in mind that we are in the unmixed case.

\section{Galois coverings of Riemann surfaces}\label{Riemannsurf}

In this section we  recall some principles of group actions on 
compact Riemann surfaces that we will use throughout the article. 
 In the first part we state
 \textit{Riemann's existence theorem}, 
in the second part the \textit{Chevalley-Weil formula} which describes the induced 
group action on the space of holomorphic $1$-forms.

\subsection{Group actions on Riemann surfaces}

Let $C$ be a compact Riemann surface,  $G$ be a finite group and 
$\psi\colon G \longrightarrow \Aut(C)$
be a group action of $G$ on $C$.
The homomorphism $\psi$ factors 
\[
\begin{xy}
  \xymatrix{
      G/K \ar[r]^{\overline{\psi}\quad} &  \Aut(C ) \\
      G \ar[ur]_{\psi} \ar[u]^{\pi}   \\				
							}
\end{xy},  
\]
where $K:=\ker(\psi)$ is the kernel of the action and $\pi$ the quotient map.
Via $\overline\psi$, we  consider $G/K$ as a subgroup of $\Aut(C)$.
For $p \in C$,  $G_p$ (resp. $\left(G/K\right)_p$) denotes the stabilizer group of $p$ for  the
$G$ (resp. $G/K$) action.

\begin{proposition}\label{nonfaithstab} \ 
\begin{itemize}
\item[i)] The group $(G/K)_p$ is cyclic for any point $p \in C$.
\item[ii)] $G_p= \langle g \rangle \cdot K$, 
for any $g \in G$ such that $ \langle \pi(g) \rangle =(G/K)_p$. 
\item[iii)] $G_{\psi(h)(p)}=hG_ph^{-1}=\langle hgh^{-1} \rangle \cdot K$ for any  $h \in G$. 
\end{itemize}
\end{proposition}

\begin{proof}
The first assertion is well known (cf.~\cite[Proposition III.3.1]{mir}).
The other two statements follow immediately, since $K=\ker(\psi) \trianglelefteq G$.
\end{proof}

\begin{definition}\label{DefGV}
Let $2 \leq m_1 \leq \ldots \leq m_r$ and  $g' \geq 0$ be integers and $H$ be a finite group.  
A \textit{generating vector} for $H$  of type $[g';m_1, \ldots ,m_r]$ is a $(2g'+ r)$-tuple 
$(d_1,e_1, \ldots , d_{g'},e_{g'};h_1, \ldots , h_r)$ of elements of $H$ such that: 
\begin{itemize}
\item[i)]
$ H = \langle d_1,e_1, \ldots , d_{g'},e_{g'},h_1, \ldots , h_r\rangle$,
\item[ii)]
$\displaystyle{\prod_{i=1}^{g'}[d_i,e_i]\cdot h_1 \ldots  h_r=1_H}$,
\item[iii)]
$\mathrm{ord}(h_i)= m_i$  for all $1 \leq i \leq r$. 
\end{itemize}
\end{definition}

The geometry behind this definition is 
\textit{Riemann's existence theorem} (cf.~\cite[Sections III.3 and III.4]{mir}).

\begin{theorem}[Riemann's existence theorem]\label{RsET}
A finite group $H$ acts faithfully and  biholomorphically on some compact Riemann surface $C$ of genus $g(C)$, 
if and only if  there exists a generating vector $(d_1,e_1, \ldots , d_{g'},e_{g'};h_1, \ldots , h_r)$ 
for  $H$ of type $[g';m_1,\ldots,m_r]$,
such that  the  Hurwitz formula holds:
\begin{equation*} 
2g(C) -2 = |H| \bigg( 2g'-2+\sum_{i=1}^r \frac{m_i - 1}{m_i} \bigg). 
\end{equation*}

In  this case $g'$ is the genus of the quotient Riemann surface $C' := C/H$
and the $H$-cover $C\to C'$ is branched in $r$ points $\{x_1, \ldots , x_r\}$ with
branching indices $m_1, \ldots ,m_r$, respectively.
Moreover, the cyclic groups $\langle h_1\rangle, \ldots , \langle h_r\rangle$ and
their conjugates provide the non-trivial stabilizers of the action of $H$ on $C$.
\end{theorem}

\begin{definition}
Let $G$ be a finite group and $K$ a normal subgroup. Let 
\[
V:=(d_1,e_1, \ldots ,d_{g'},e_{g'};h_1, \ldots , h_r)
\] 
be a generating vector for $G/K$. 
Let $g_i \in G$ be a representative for the element $h_i$ in the generating vector $V$. 
The \textit{stabilizer set of $V$} is defined as 
\[
\Sigma_{V}:= 
\left(\bigcup_{g \in G}\bigcup_{i \in \mathbb Z}\bigcup_{j = 1}^r\left \{  gg_j^ig^{-1}\right\} 	
\cup \{1_G\}
\right) \cdot K
\,. 
\]
\end{definition}

According to Proposition \ref{nonfaithstab} and Riemann's existence theorem, it holds:

\begin{corollary}\label{stabsetV}
Let $C$ be a compact Riemann surface, $\psi: G \to \Aut(C)$  a group action with kernel $K$ and 
$V$  a generating vector for $G/K$ associated to $\overline{\psi}: G/K \to \Aut(C)$. Then: 
\[
\Sigma_{V}= \bigcup_{p \in C} G_p\,.
\]
\end{corollary}

\subsection{The Chevalley-Weil formula}\label{ChevWeil}

Let $C$ be a compact Riemann surface, 
$G$ be a finite group,  $\psi\colon G \rightarrow \Aut(C) $ be a group  action  with  kernel 
$K:=\ker(\psi)$ and  $\pi\colon G\to G/K$ be the quotient map.

Let $\varphi$ be the representation of $G$ associated to $\psi$ via pull-back of holomorphic $1$-forms: 
\begin{equation}\label{rep}
\varphi\colon G \to \mathrm{GL}\left( H^{1,0}(C) \right)\,,  \quad  g \mapsto [ \omega \mapsto \psi(g^{-1})^{\ast}\omega]\, 
\end{equation}
and  $\chi_{\varphi}$ its character. The aim of this subsection is to describe this character. 
For details about character theory and complex linear 
representations we refer to \cite{Isaacs76}.

Since the following diagram commutes
\[
\begin{xy}
  \xymatrix{
      G/K \ar[r]^{\overline{\varphi}\quad} &  \mathrm{GL}\big(H^{1,0}(C)\big)  \\
      G \ar[ur]_{\varphi}  \ar[u]^{\pi}   \\				
							}
\end{xy},
\]
the character of $\varphi$ is determined by the character of  $\overline{\varphi}$.
For this reason it suffices to consider the case $K=\{1_G\}$ and 
we identify $G$ with $\psi(G) \leq \Aut(C)$.

\begin{remark}\label{orth} The set of irreducible characters $Irr(G)$ of a finite group $G$ is an orthonormal 
basis of the vector space of class functions $CF(G)$
of the group $G$ with respect to the Hermitian product 
\begin{equation*}
\langle\alpha, \beta \rangle:=\frac{1}{|G|}\sum_{g\in G}\alpha(g)\overline{\beta(g)}, 
\quad \quad \makebox{ $\alpha, \beta \in CF(G)$}.
\end{equation*} 
\end{remark}

\noindent The character $\chi_{\varphi}$ has the decomposition
\[
\chi_{\varphi} = \sum_{\chi \in Irr(G)} \langle \chi_{\varphi}, \chi\rangle \cdot \chi.
\]
The Chevalley-Weil formula provides a way to compute the integral coefficients 
$\langle\chi_{\varphi},\chi \rangle$.

\begin{definition}
Let $(d_1,e_1, \ldots , d_{g'},e_{g'};h_1, \ldots , h_r)$
be a generating vector  of type $[g';m_1, \ldots ,m_r]$ for $G$ and 
$\varrho \colon G \to \mathrm{GL}(W)$ be a representation.
Let $1\leq i \leq r$ and   $0 \leq \alpha \leq m_i-1$,
we denote by 
$N_{i,\alpha}$ the multiplicity of $ \exp\Big(\frac{2\pi \sqrt{-1} }{m_i}\Big)^\alpha$ as eigenvalue of  $\varrho(h_i)$.
\end{definition}

\begin{theorem}[Chevalley-Weil Formula, cf.~\cite{ChevWeil}]
Let 
$\psi\colon G \rightarrow \Aut(C)$
be a faithful group action of a finite group $G$ on a compact Riemann surface $C$ and 
$
(d_1,e_1, \ldots , d_{g'},e_{g'};h_1, \ldots , h_r)
$
be an associated generating vector of type $[g';m_1, \ldots ,m_r]$. Then, 
for an irreducible representation 
$\varrho$  of $G$ with character 
$\chi_\varrho$ and degree $d$, it holds  
\begin{equation*}
\langle\chi_\varrho, \chi_{\varphi}\rangle 
= d(g'-1) + \sum_{i=1}^r \sum_{\alpha=1}^{m_i-1} \frac{\alpha \cdot N_{i, \alpha}}{m_i}  +  \langle \chi_{\varrho}, \chi_{triv}  \rangle,
\end{equation*}
where $\chi_{triv}$ is the trivial character.  
\end{theorem}
 \begin{proof}
  Let $L$ be the class function $\chi_{triv} -\chi_\varphi$.
By bilinearity 
$ \langle \chi_{\varrho}, \chi_{\varphi}  \rangle   = - \langle \chi_{\varrho}, L \rangle +
\langle \chi_{\varrho}, \chi_{triv}  \rangle$ and the inner product $\langle  \chi_{\varrho},  L \rangle$ is
\begin{equation*}
\langle  \chi_{\varrho},  L \rangle = \frac{1}{|G|} \sum_{g \in G } \chi_{\varrho}(g) \overline{L(g)} = 
\frac{1}{|G|} \chi_{\varrho}(\mathrm{1_G}) \overline{L(\mathrm{1_G})} + 
\frac{1}{|G|}\sum_{\substack{g \in G \\ g \ne \mathrm{id}}} \chi_{\varrho}(g) \overline{L(g)}. 
\end{equation*}
The value  $\chi_{\varrho}(\mathrm{1_G})$ is the degree $d$ of  $\varrho$ and 
 $L(\mathrm{1_G}) = 1- g(C)$.
On the other hand, 
\[
L(g)= 1-\chi_\varphi(g) =\sum_{ p \in Fix(g)} \dfrac{1}{1 -  J_p(g)} \quad \makebox{for all} \quad g \neq 1_G\,,
\] 
where  the latter  equality is the Eichler trace formula (see \cite[Theorem V.2.9]{FK80}) and 
$ J_p(g)$ denotes the Jacobian of the automorphism $g$ at the point $p$.


Since $\overline{L(g)} =L(g^{-1})$, it follows 
\[
\frac{1}{|G|}\sum_{\substack{g \in G \\ g \ne \mathrm{id}}}\chi_{\varrho}(g) \overline{L(g)} = 
\dfrac{1}{|G|}\sum_{\substack{g \in G \\ g \ne \mathrm{id}}} \sum_{ p \in Fix(g)} \dfrac{\chi_{\varrho}(g)}{1 -  J_p(g^{-1})}
=\dfrac{1}{|G|}\sum_{p \in Fix(C) } \sum_{\substack{g \in G_p \\ g \ne \mathrm{id}}} \frac{\chi_{\varrho}(g)}{1 -  J_p(g^{-1})},
\]
where $Fix(C):=\{p\in C\mid G_p\neq \{1_G\}\}$. 

Let $f\colon C\to C/G$ be the quotient map  and $\{x_1,\ldots, x_r\} $ be the 
branch locus. For all $1 \leq i \leq r$, there exists a point  $p_i \in f^{-1}(x_i)$ with 
$G_{p_i}= \langle h_i \rangle$  and  for any $g \in G$ we have $G_{g(p_i)}=\langle g h_i g^{-1}\rangle$;
moreover, every $p\in Fix(C)$ maps to a branch point.
As shown in \cite[4.7.2 Satz]{lamotke}, the element $h_i$ is the unique element in $G_{p_i}$ such that  
\[
J_{p_i}(h_i)= \exp\left(\frac{2\pi \sqrt{-1} }{m_i}\right)=: \xi_{m_i}\,,
\]
hence the Jacobian  $J_{g(p_i)}(gh_i^l g^{-1})$ equals $\xi_{m_i}^l$ for  $l \in \mathbb{Z}$. 
Since $\chi_{\varrho}$ is a class function  $\chi_{\varrho}(h_i^l)= \chi_{\varrho}(gh_i^lg^{-1})$. 
This implies 
\[
\dfrac{1}{|G|}\sum_{p \in Fix(C) } \sum_{\substack{g \in G_p \\ g \ne \mathrm{id}}}\frac{\chi_{\varrho}(g)}{1 -  J_p(g^{-1})}
=
\frac{1}{|G|} \sum_{i=1}^r \frac{|G|}{m_i} \sum_{l=1}^{m_i-1} \frac{\chi_\varrho(h_i^l) }{1 - \xi_{m_i}^{-l}}
=
 \sum_{i=1}^r \frac{1}{m_i} \sum_{l=1}^{m_i-1} \frac{\sum_{\alpha = 0}^{m_i-1} 
N_{i,\alpha}\cdot (\xi_{m_i}^{\alpha})^l}{1 - \xi_{m_i}^{-l}}.
\]

The last member of the equation above can be simplified using 
the following equality  for  Fourier-Dedekind sum (c.f. \cite[Equation 1.8]{BR07}):
\[
\sum_{l=1}^{m_i-1} \frac{ (\xi_{m_i}^{l})^ \alpha}{1 - \xi_{m_i}^{-l}} = \frac{m_i -1}{2} - \alpha 
\quad \makebox{for all} \quad 0 \leq  \alpha  \leq m_i - 1, 
\]
whence  
\[
\langle \chi_{\varrho}, L \rangle = \frac{1}{|G|} d (1- g(C)) +  \sum_{i=1}^r \sum_{\alpha =0}^{m_i-1} 
\frac{N_{i,\alpha}}{m_i} \bigg(\frac{m_i -1}{2} - \alpha \bigg)\,.
\]
In combination with Hurwitz' formula we obtain 
$$
\langle \chi_{\varrho}, \chi_{\varphi}  \rangle =    \langle \chi_{\varrho}, \chi_{triv}  \rangle  + d(g' - 1 ) +\sum_{ i = 1 }^r
\left[ d \bigg(\frac{m_i-1}{2 m_i}\bigg) - 
\sum_{\alpha =0}^{m_i-1} N_{i,\alpha}\bigg(\frac{m_i -1}{2m_i} - \frac{\alpha}{m_i} \bigg)\right]\,.
$$

The proof of the  Chevalley-Weil formula is finished because 
the degree of the representation $\varrho$ is $d = \sum_{\alpha =0}^{m_i-1} N_{i,\alpha}$, for all $1 \leq i \leq r$. 
\end{proof}

\begin{comprem}
For a given finite group $G$, it is in general very difficult to determine all its 
irreducible representations even using a computer algebra system, e.g. MAGMA (\cite{magma}).
MAGMA is able to compute the irreducible representations of $G$ only over fields of positive characteristic and,
if $G$ is solvable, over cyclotomic fields. 
However MAGMA can determine the \textit{character table} of any  sufficiently small finite group. This is indeed  
enough to apply the Chevalley-Weil formula according to  
the following classical argument. 

Let $\varrho \colon H \to \mathrm{GL}(W)$ be a representation of degree $n$ and
$h\in H$ an element of order $m$.
To determine the eigenvalues 
$\lambda_1,\ldots, \lambda_n$ of $\varrho(h)$,
it suffices to know the character $\chi_\varrho$. \\
The first observation is 
\[
\chi_\varrho(h^k) =\lambda_1^k + \ldots  + \lambda_n^k.
\]
Recall that the power sum polynomials $p_k:=\sum_{i=1}^n x_i^k$ are 
related to the \textit{elementary symmetric polynomials} $s_i$
via the \textit{Newton identities}: 
\[
k\cdot s_k= \sum_{i=1}^k (-1)^{i-1} s_{k-i} p_i\,, \quad 1 \leq k \leq n.
\]
Hence we can compute the values $s_k(\lambda_1, \ldots ,\lambda_n)$ 
recursively from $\chi_\varrho(h), \ldots ,\chi_\varrho(h^k)$ and the previous ones  
\[
s_0(\lambda_1, \ldots ,\lambda_n), \ldots , s_{k-1}(\lambda_1, \ldots ,\lambda_n). 
\]
In this way we obtain the characteristic polynomial $f_h$ of  $\varrho(h)$: 
\[
f_h(x)=\sum_{k=0}^n a_k x^k,  \quad \quad 
a_k= (-1)^{n-k} s_{n-k}(\lambda_1,\ldots, \lambda_n). 
\]
Finally, since the roots of $f_h(x)$ are powers of $\xi_{m}$,
one can  factorize $f_h(x)$ and determine the spectrum of $\varrho(h)$.
\end{comprem}

\section{Group theoretical description and Invariants of Threefolds  Isogenous to a Product}\label{group_descr}
 
The aim of this section is to give a \textit{group theoretical description} of 
threefolds isogenous to a product. In the first part we
introduce the \textit{algebraic datum} of a threefold isogenous to a product $X$. 
Based on this definition we show how to compute the \textit{Hodge diamond} and the \textit{fundamental group} of $X$
from the algebraic datum.

\begin{definition}\label{def_min_real}
Let $X:=(C_1 \times C_2 \times C_3)/G$ be a  threefold isogenous to a product.
\begin{itemize}
\item[i)]
The action of $G$ is \textit{minimal} if 
$K_i \cap K_j = \lbrace 1_G \rbrace$ for all  $1 \leq i < j \leq 3$.
\item[ii)]
The action of $G$  is \textit{absolutely faithful} if  $K_i=\{1_G\}$ for all $1 \leq i \leq 3$. 
\end{itemize}
\end{definition}

\begin{proposition}\label{faithfulassum}\cite[p.~105]{Cat08},
Every threefold  isogenous to a product 
admits a unique minimal realization.
\end{proposition}

From now on, when we consider a realization $\big(C_1 \times C_2 \times C_3\big)/G$ of a threefold isogenous to a product $X$, 
we assume it is minimal. However, it is not possible to assume that the action of $G$  is {absolutely faithful}.

We shall associate to a threefold isogenous to a product $X$ certain algebraic and numerical data. According to Riemann's existence theorem, to each action $G/K_i\to \Aut(C_i)$ corresponds
a generating vector $V_i$ for $G/K_i$ of type $T_i$; note that only the type $T_i$ is uniquely determined.  
By Corollary \ref{stabsetV} the stabilizer set  
$\Sigma_{V_i}$ is the set of element of $G$ admitting at least one fixed point on $C_i$.
Therefore the freeness of the $G$-action on the product is reflected by the condition that the triple 
$(V_1,V_2,V_3)$ is \textit{disjoint}: 
\[
\Sigma_{V_1} \cap \Sigma_{V_2} \cap \Sigma_{V_3} = \lbrace 1_G \rbrace.
\]

\begin{definition} \label{algebraicdata}
To the threefold $X$ we attach the tuple $(G, K_1, K_2, K_3, V_1, V_2, V_3)$,
and  call it an  \textit{algebraic datum} of $X$.
The  \textit{numerical datum} of $X$ is the tuple 
\[
(|G|, |K_1|, |K_2|, |K_3|, T_1, T_2, T_3).
\]
\end{definition}

\begin{remark}\label{converse}
Let $G$ be a finite group and $K_i \trianglelefteq G$ be normal subgroups such that
$K_i\cap K_j=\{1_G\}$ for all $1\leq i <j \leq 3$. Let  $(V_1,V_2,V_3) $ be a disjoint triple of 
generating vectors, where $V_i$ is a generating vector for $G/K_i$.
Then, by Riemann's existence theorem, there exists a threefold isogenous to a product
with  algebraic datum $(G, K_1, K_2, K_3, V_1, V_2, V_3)$.
\end{remark}

Next we explain how to compute the Hodge diamond of a threefold isogenous to
a product $X$ from an algebraic datum of $X$.

\begin{proposition}\label{invar}
Let $Y$ be a projective manifold and $G$ be a finite group of automorphisms acting freely on $Y$, then 
\[
H^{p,q}(Y/G) \simeq  H^{p,q}(Y)^G.
\]
\end{proposition}
\begin{proof}
The quotient map $\pi\colon Y \to Y/G$ is a finite unramified  covering.
By \cite[Proposition 3G.1]{Hatcher} we have an isomorphism 
$\pi^*\colon H^k(Y/G,\mathbb C)\stackrel{\sim}{\longrightarrow} H^k(Y,\mathbb C)^G$.
Since a holomorphic map induces a morphism of Hodge structures (c.f. \cite[Section 7.3.2]{Voisin07}),
considering the Hodge decomposition on both sides,  we get a diagonal isomorphism 
\[\pi^*\colon\bigoplus_{p+q=k}H^{p,q}(Y/G) \stackrel{\sim}{\longrightarrow}  \bigg(\bigoplus_{p+q=k}H^{p,q}(Y)\bigg)^G=
 \bigoplus_{p+q=k}H^{p,q}(Y)^G\,,\]
and the statement follows.
\end{proof}

The proposition above motivates the following definition.

\begin{definition}
Let $X:=\big(C_1 \times C_2 \times C_3\big)/G$  be a threefold isogenous to a product.  
We define representations of $G$ via pull-back:
\[
\phi_{p,q} \colon G  \to  \mathrm{GL}\big(H^{p,q}(C_1 \times C_2 \times C_3)\big), \quad 
 g \mapsto  [\omega \mapsto (g^{-1})^{\ast}\omega].
\]
The characters of $\phi_{p,q}$ are denoted by $\chi_{p,q}$. 
\end{definition}

As a direct consequence of Proposition \ref{invar} the Hodge numbers $h^{p,q}(X)$ are
given by $\langle  \chi_{p,q}, \chi_{triv}  \rangle$.

The group actions $\psi_i\colon G \to \Aut(C_i)$ induce representations 
\[
\varphi_i\colon G \rightarrow \mathrm{GL}\big(H^{1,0}(C_i) \big)   
\quad  g \mapsto [ \omega \mapsto \psi_i(g^{-1})^{\ast}\omega]\,
\]
with characters $\chi_{\varphi_i}$, see Section \ref{ChevWeil}.

\begin{theorem}\label{KunnChar}
For the characters $\chi_{p,q}$ it holds:
\begin{itemize}
\item[i)]
$\chi_{1,0} = \chi_{\varphi_1} + \chi_{\varphi_2} + \chi_{\varphi_3}$,
\item[ii)]
$\chi_{1,1} = 2 \mathfrak Re (\chi_{\varphi_1} \overline{\chi_{\varphi_2}} + \chi_{\varphi_1} \overline{\chi_{\varphi_3}}
+ \chi_{\varphi_2} \overline{\chi_{\varphi_3}}) + 3\chi_{triv}$,
\item[iii)]
$\chi_{2,0} = \chi_{\varphi_1} \chi_{\varphi_2} + \chi_{\varphi_1}  \chi_{\varphi_3} + \chi_{\varphi_2}  \chi_{\varphi_3}$,
\item[iv)]
$\chi_{2,1} = \overline{\chi_{\varphi_1}}  \chi_{\varphi_2}  \chi_{\varphi_3}  +  \chi_{\varphi_1}  \overline{\chi_{\varphi_2}}  \chi_{\varphi_3}
+ \chi_{\varphi_1}  \chi_{\varphi_2}  \overline{\chi_{\varphi_3}} + 2(\chi_{\varphi_1} + \chi_{\varphi_2} + \chi_{\varphi_3})$,
\item[v)]
$\chi_{3,0} = \chi_{\varphi_1}  \chi_{\varphi_2}  \chi_{\varphi_3}$.
\end{itemize}
\end{theorem}

\begin{proof}
According to  K\"unneth's formula \cite[p.103-104]{G-H}: 
\[
H^{p,q}(C_1 \times C_2 \times C_3)= \bigoplus_{\substack {s_1+s_2+s_3=p \\ t_1 + t_2+t_3 =q}}
H^{s_1,t_1}(C_1) \otimes H^{s_2,t_2}(C_2) \otimes H^{s_3,t_3}(C_3).
\]
Let $\omega=\omega_1 \otimes \omega_2 \otimes \omega_ 3$ be a pure tensor contained in a direct summand of this decomposition.
Since the action of $G$ is diagonal, 
the tensors $\omega$ and $(g^{-1})^{\ast}\omega$ are in the same summand for all $g \in G$.  
This implies that each summand is a subrepresentation of $H^{p,q}(Y)$  and the character
$\chi_{p,q}$ is the sum of the characters of these subrepresentations.
By definition of $\chi_{\varphi_i}$ and the fact that the character of a tensor product is the product of the characters, the statement follows.
\end{proof}

According to the theorem above the characters $\chi_{p,q}$ are given in terms of the characters 
$\chi_{\varphi_i}$ which can be computed from an algebraic datum of $X$ using 
the  Chevalley-Weil formula. 
As a by-product the Hodge numbers of $X$ can be determined from an algebraic datum of $X$.

In the last part of this section we briefly explain how to compute the fundamental group  of a threefold  
isogenous to a product $X$ 
from an algebraic datum of $X$.
This topic is treated in greater generality in \cite{DP10} 
and we refer to it for further details.

\begin{definition}\label{gv}
Let  $g'\geq 0 \mbox{ and  } m_1, \ldots, m_r \geq 2$
be integers.
The orbifold surface group of type $T=[g';m_1, \ldots, m_r]$ is defined as: 
\[
\mathbb{T}(T):=
\langle a_1,b_1,\ldots, a_{g'},b_{g'},c_1, \ldots, c_r ~ \big\vert ~
 c_1^{m_1}, \ldots, c_r^{m_r},\prod_{i=1}^{g'} [a_i,b_i]\cdot c_1 \cdots c_r\rangle. 
\]
\end{definition}

\begin{lemma}
Let  $X:=(C_1\times C_2 \times C_3)/G$ be  a threefold isogenous to a product.
The fundamental group $\pi_1(X)$
sits in the following short exact sequence:
$$1\rightarrow\pi_1(C_1)\times\pi_1(C_2)\times \pi_1(C_3)\rightarrow \pi_1(X)\rightarrow G\rightarrow 1\,.$$
\end{lemma}

Let $(G, K_1, K_2, K_3, V_1, V_2, V_3)$ be an algebraic datum of a threefold isogenous to product $X$.
The generating vectors $V_i$ of $H_i:=G/K_i$ determine a homomorphism $\gamma_i\colon \mathbb T(T_i) \to H_i$. 
The kernel of $\gamma_i$ is isomorphic to $\pi_1(C_i)$:
\[
1\longrightarrow \pi_1(C_i) \longrightarrow \mathbb T(T_i)\stackrel{\gamma_i}{\longrightarrow} H_i \longrightarrow 1 \,. 
\]
Let $\Gamma_i:= G\times_{H_i} \mathbb T(T_i)$ be the fibre product, 
 $\zeta_i\colon \Gamma_i \to G$ be the projection on the first factor and define 
\[
\mathbb G:=\lbrace(x,y,z)\in \Gamma_1\times \Gamma_2 \times \Gamma_3\mid \zeta_1(x)=\zeta_2(y)=\zeta_3(z)\rbrace.
\]
There is the following
description of the fundamental group and the first homology group of $X$.

\begin{proposition}[{c.f.\cite[Proposition 3.3]{DP10}}]
Let $X$ be a threefold isogenous to a product. Then 
$\pi_1(X)\cong \mathbb G$ and $H_1(X, \mathbb Z)\cong \mathbb G^{ab}$. 
\end{proposition}

\section{Combinatorics, Bounds and the Algorithm}\label{bounds_smooth}

In the first part of this section we derive constraints on the numerical datum of a threefold isogenous to a product. 
These constraints imply that there are only finitely many numerical data of threefolds isogenous to 
a product of unmixed type with a fixed value of  
$\chi(\mathcal{O}_X)$ and consequently only finitely many algebraic data for such threefolds. 
In the second part we give an algorithm searching systematically through all possibilities  
and thereby classifying all threefolds
isogenous to a product of unmixed type with a fixed value of $\chi(\mathcal O_X)$. Running a 
MAGMA implementation of our algorithm in the absolutely faithful case for $\chi(\mathcal O_X)=-1$ we obtain the 
main theorem of this article.

\begin{definition}
Let  $g'\geq 0$ and  $m_1, \ldots , m_r \geq 2$
be integers.  We associate to the tuple  $T:=[g';m_1, \ldots ,m_r]$ the rational number
\[
\Theta(T):=2g'-2+\sum_{i=1}^r \frac{m_i-1}{m_i}\,.
\]
\end{definition}

\begin{remark}[{cf.\, \cite[Lemma III.3.8]{mir}}]\label{estimTheta} 

If $\Theta(T)>0$, then $\Theta(T)\geq f(g')$, where:
\[
f(g'):=\begin{cases}
  1/42,  & \text{if} \quad g'=0\\
  1/2, & \text{if} \quad g'=1\\
	2g'-2, & \text{if}  \quad g' \geq 2\\
\end{cases}
\]
Moreover, $\Theta(T)=f(g')$ if and only if $T \in \big\lbrace [0;2,3,7], [1;2], [2g'-2;-] \big\rbrace$.

\end{remark}

\begin{remark} \label{bre}
According to Hurwitz' Theorem, the order of the automorphism group of a compact Riemann surface $C$ of genus $g(C)\geq 2$ 
is bounded:
\[
|\Aut(C)| \leq 84\big(g(C)-1\big). 
\]
It is known that this bound is not sharp in general.
However, if  
$2 \leq g(C) \leq 301$, sharp bounds for $|\Aut(C)|$ can be found in \cite[p.~87]{Conder}.
\end{remark}

\begin{proposition}\label{boundgroupord}
Let $X:=(C_1\times C_2\times C_3)/G$ be a threefold isogenous to a product with numerical datum 
$(|G|, |K_1|, |K_2|, |K_3|, T_1, T_2, T_3)$.
Then 
\[
|G| \leq \bigg\lfloor \sqrt{- 8 \cdot \chi(\mathcal O_X) \prod_{i=1}^3\frac{|K_i|}{f(g_i')}} \bigg\rfloor.
\]
 \end{proposition}

\begin{proof}
Since $g(C_i) \geq 2$ Hurwitz' formula 
\[\displaystyle{g(C_i)-1 = \frac{1}{2}\frac{|G|}{|K_i|} \Theta(T_i)}\]
implies $\Theta(T_i)>0$. In combination with Proposition \ref{invarsmooth} and Remark \ref{estimTheta} we can estimate 
\[
-\chi(\mathcal O_X)= \frac{1}{|G|} \prod_{i=1}^3 \big(g(C_i)-1\big) =\frac{1}{8|G|}\prod_{i=1}^3 \frac{|G|}{|K_i|} \Theta(T_i) 
\geq \frac{|G|^2}{8} \prod_{i=1}^3 \frac{f(g_i')}{|K_i|}\,.
\]
\end{proof}

\begin{corollary}
Let $X$ be a threefold isogenous to a product. If the
action of $G$ is absolutely faithful, then
\[
|G| \leq \lfloor 168 \sqrt{-21 \chi(\mathcal O_X)} \rfloor.
\]
\end{corollary}

In the general case, we also have  that $|G|$ is bounded, unfortunately this bound is very large: 

\begin{proposition}
Let $X$ be a threefold isogenous to a product, then
\[
|G| \leq 84^6 \cdot \chi(\mathcal O_X)^2. 
\]
\end{proposition}

\begin{proof}
We write down the formula for $\chi(\mathcal O_X)$ in the following way
$$
-\chi(\mathcal O_X) =\frac{(g(C_1)-1)}{|K_3|}  \frac{(g(C_2)-1)}{|K_3|} \frac{(g(C_3)-1)|K_3|}{|G|} |K_3|
\geq \left(\frac{1}{84}\right)^3  |K_3|\,.
$$
The inequality is Hurwitz' theorem. 
It holds because $K_3$ acts faithfully on $C_1$ and $C_2$, by the minimality assumption, 
and $G/K_3$ acts faithfully  on $C_3$. 
The same bound holds for $|K_1|$ and $|K_2|$. The claim follows combining these inequalities with 
Proposition \ref{boundgroupord}.
\end{proof}

\begin{remark}
It would be interesting to understand if there exists a significantly 
better bound for $|G|$ in terms of  $\chi(\mathcal{O}_X)$. 
\end{remark}

 \begin{proposition}\label{bounds}
Let $X:=(C_1\times C_2\times C_3)/G$ be a  threefold isogenous to a product with numerical datum 
$(|G|, |K_1|, |K_2|, |K_3|, T_1, T_2, T_3)$, where $T_i:=[g_i'; m_{i,1},\ldots, m_{i,r_i}]$.
Then
\begin{itemize}\setlength{\itemsep}{.3\baselineskip}
\item[i)] $|K_i| ~ \big\vert  ~ \big( g(C_{[i+1]}) - 1 \big) \big( g(C_{[i+2]}) - 1 \big)$,
\item[ii)] $m_{i,j} ~ \big\vert ~  \big( g(C_{[i+1]}) - 1 \big) \big( g(C_{[i+2]}) - 1 \big)$,
\item[iii)]  $\big( g(C_i)- 1 \big) ~ \big\vert ~  \chi(\mathcal O_X) \frac{|G|}{|K_i|}$,
\item[iv)] $r_i \leq \dfrac{4\big( g(C_i) - 1 \big) |K_i|}{|G|} - 4g_i' + 4$,
\item[v)] $ m_{i,j} \leq  4g(C_i) + 2$,
\item[vi)]
$\displaystyle{ g_i' \leq 1  - \frac{|K_i| ~ 
\chi(\mathcal O_X)}{\big( g(C_{[i+1]}) - 1 \big) \big( g(C_{[i+2]}) - 1 \big)} \leq 1  -  \chi(\mathcal O_X)}$. 
\item[vii)]
$|K_i| \cdot |K_j|  ~ \big\vert  ~ |G|$, \quad  for all $1 \leq i < j \leq 3$.

\end{itemize}
Here $[\,\cdot \,]$ denotes the residue $\mathrm{mod }\ 3$. 
\end{proposition}

\begin{proof} We assume that $i=1$.

i-ii)  Let $V_1:=(d_1,e_1, \ldots , d_{g_1'}, e_{g_1'}; h_1, \ldots , h_{r})$ be a 
generating vector of type $T_1$ associated to the covering 
\[
f_1\colon C_1\rightarrow C_1/ (G/K_1) 
\]  
Let $g_j \in G$ be a representative of  $h_j \in G/K_1$ (if $r=0$, then  we set $g_j:=1_G$).
By the minimality of the $G$-action, the subgroup $\langle g_j\rangle \cdot K_1 \leq G$
acts  faithfully on $C_2 \times C_3$. 
Furthermore, $\langle g_j\rangle \cdot K_1$ is contained in $\Sigma_{V_1}$,
therefore, it acts
freely on $C_2 \times C_3$. In other words
\[
S:=\frac{C_2 \times C_3}{\langle g_j\rangle \cdot K_1}
\]
is a surface isogenous to a product with holomorphic Euler-Poincar\'e characteristic  
\[
\chi(\mathcal O_S) =\frac{\big( g(C_2) - 1 \big) \big( g(C_3) - 1 \big)}{| \langle g_j \rangle \cdot K_1|}\,.
\]
We conclude that  
\[
|K_1|  ~ \big\vert ~ \big( g(C_2) - 1 \big) \big( g(C_3) - 1 \big) \quad \makebox{and} \quad 
m_{1,j}= \mathrm{ord}(h_j) ~ \big\vert ~ \big( g(C_2) - 1 \big) \big( g(C_3) - 1 \big).
\]

iii) The statement follows from part i) and Proposition \ref{invarsmooth}:
\[
-\chi(\mathcal O_X)\frac{|G|}{|K_1|} = 
\big( g(C_1)-1 \big) \frac{\big( g(C_2) - 1 \big) \big(g(C_3) - 1\big)}{|K_1|} \,.
\]

iv) This is a straightforward consequence of Hurwitz' formula, using the fact  $m_{1,j}\geq 2$.

v) For a cyclic group $H$ acting faithfully on a 
compact Riemann surface $C$ of genus $g(C)\geq 2$,  Wiman's bound (see \cite{wiman}) holds: 
\[
|H| \leq 4g(C) + 2\,.
\] 
In particular $m_{1,j}\leq  4g(C_1) + 2$.

vi) According to Proposition \ref{invarsmooth} and  Hurwitz' formula
\[
g_1' -1  \leq \frac{\Theta(T_1)}{2}=\frac{|K_1|}{|G|}\big( g(C_1)-1 \big) =
  \frac{-|K_1| \chi(\mathcal O_X)}{\big( g(C_2) - 1 \big) \big( g(C_3) - 1 \big)}\,. 
\]
The second inequality follows now from part i).

vii) From the minimality of the $G$-action it follows $K_i \times K_j =K_i \cdot K_j  \leq G $.
\end{proof}

\begin{corollary}
Let $k$ be a negative integer, then there exists only finitely many 
algebraic data of threefolds $X$ isogenous to a product 
of unmixed type with $\mathbb \chi(\mathcal O_X)=k$.
\end{corollary}

Using the above combinatorial constraints we can give an algorithm to classify threefolds isogenous to a product in the unmixed case.
An implemented MAGMA version of the algorithm can be downloaded from:
\[
\makebox{\url{http://www.staff.uni-bayreuth.de/~bt300503/}}.
\]

 \begin{Input} The  holomorphic {Euler-Poincar\'e-characteristic} $\chi $ and
 the group order $\texttt n$.
\end{Input}

\begin{step} Compute the possible orders of the kernels, i.e. the triples $(k_1,k_2,k_3) \in \mathbb N^3$
satisfying the conditions imposed by Proposition \ref{bounds}.

\end{step}

\begin{step} Determine the possible genera of the curves $C_i$ and $C'_i=C_i/G$: 
for every triple in the output of Step 1,
construct the set of $9$-tuples
\[
(k_1,k_2,k_3,g_1,g_2,g_3,g_1',g_2',g_3'),
\] 
which fulfil the conditions  given by
Remark \ref{bre},  Proposition \ref{boundgroupord} and  Proposition \ref{bounds}.

\end{step}

\begin{step}  For every $9$-tuple in the output of Step 2 construct the set of  $9$-tuples
\[
(k_1,k_2,k_3,g_1,g_2,g_3,T_1,T_2,T_3),
\] 
 where $T_i=[g'_i; m_{i,1}, \ldots,m_{i,r_i}]$ are the types which satisfy the conditions of  Proposition \ref{bounds}.

\end{step}

\begin{step} 
For every $9$-tuple in the output of the Step 3 search among the groups of order $\texttt{n}$
for groups $G$ containing normal subgroups $K_i$ of order $k_i$ which have pairwise trivial
intersection:
\[
K_i \cap K_j = \lbrace 1_G \rbrace\, , \quad 1 \leq i< j \leq 3\,.
\]

 For every triple of subgroups satisfying these conditions  search for triples $(V_1,V_2,V_3)$ of disjoint generating vectors, where $V_i$ is a generating
vector for $G/K_i$ of type  $T_i$.

\end{step}

\begin{step} 
For each $7$-tuple $(G,K_1,K_2,K_3, V_1,V_2,V_3)$ in the output of Step 4 
there exists a threefold $X$ isogenous
to a product with this algebraic datum (see Remark \ref{converse}) and $\chi(\mathcal O_X)=\chi$ 
and $|G|=\texttt{n}$.
Using the method described in Section \ref{group_descr}, 
compute the Hodge diamond and the fundamental group of $X$.

\end{step}

\begin{Output}  The occurrences of 
\[
[G,K_1,K_2,K_3,T_1,T_2,T_3, p_g, q_2, q_1, h^{1,1}, h^{1,2}, H_1(-,\mathbb Z)]\,.
\]
for all threefolds isogenous to a product $X$
with  $\chi(\mathcal O_X)=\chi$ and $|G|=\texttt{n}$.
Note that we store $H_1$ instead of $\pi_1$ since the latter is always infinite and a presentation of this group is in general 
very long.
\end{Output}

\begin{comprem}
In Step 4, we search for generating vectors. We point out that different generating vectors may determine threefolds with the same invariants. For example,
this happens if (but not only if) they differ by some \textit{Hurwitz moves}.
These moves are described in \cite{CLP12}, \cite{Zimmermann} and \cite{penegini} and we refer to them for
further details.
\end{comprem}

\begin{comprem}\label{except}
The algorithm works for arbitrary values of $\chi$ and $\texttt{n}$, but the implemented MAGMA version has some 
technical problems.
If the output of Step 3 is not empty, then
 in Step 4 the program has to run through all groups of order $\texttt{n}$. 
Here we have to use the database of Small Groups, which contains:
\begin{itemize}
\item all groups of order up to 2000, excluding the groups of order 1024;
\item the groups whose order is a product of at most 3 primes;
\item the groups of order dividing $p^6$ for $p$  prime; 
\item  the groups of order $p^n q$, where $p^n$ is a prime-power dividing $2^8$, $3^6$, $5^5$ or $7^4$ and $q$ is a prime different from $p$.
 \end{itemize}
In the other cases Step 4 cannot be performed and this exceptional cases have to be treated separately. \\
We  remark that running through the groups of order 
 $256, 384, 512, 768, 1152, 1280, 1536$, and $ 1920$
is	  time-consuming,  because the number of these groups is high; e.g.,
there are $56092$  groups of order $256$.
\end{comprem}
 
Without additional conditions on the kernels, the bound for the group order 
\[
|G| \leq 84^6  \chi(\mathcal O_X)^2
\]
is too large to perform a complete classification, even for small values of $\chi(\mathcal O_X)$. \\
On the other hand if the group action is assumed to be absolutely faithful, then the bound becomes  
\[
|G| \leq \lfloor 168 \sqrt{-21 \chi(\mathcal O_X)} \rfloor
\]
and a full classification, at least for $\chi(\mathcal O_X)$ small, 
is feasible.
 
\begin{proof}[Proof of Theorem \ref{thm51}]
We run our MAGMA implementation in the absolutely faithful case for $\chi=-1$ 
and for all  $\texttt{n} \leq \lfloor 168 \sqrt{21} \rfloor =769$. 
The theorem follows since for each exceptional group order  the output of Step 3 is  empty. 
\end{proof}

\bibliographystyle{alpha}

\begin{thebibliography}{CCML98}

\bibitem[BCG08]{BCG08}
I.~Bauer, F.~Catanese, and F.~Grunewald.
\newblock The classification of surfaces with $p_g=q=0$ isogenous to a product
  of curves.
\newblock {\em Pure Appl. Math. Q.}, 4(2):547--586, 2008.

\bibitem[BCP97]{magma}
W.~Bosma, J.~Cannon, and C.~Playoust.
\newblock The {M}agma algebra system. {I}. {T}he user language.
\newblock {\em J. Symbolic Comput.}, 24(3-4):235--265, 1997.
\newblock Computational algebra and number theory (London, 1993).

\bibitem[Bea82]{Be82}
A.~Beauville.
\newblock L' in\'egalit\'e $p_g\geq 2q-4$ pour les surfaces de type
  g\'en\'eral.
\newblock {\em Bull. Soc. Math. France}, 110(3):343--346, 1982.
\newblock Appendix to \cite{deb82}.

\bibitem[BR07]{BR07}
M.~Beck and S.~Robins.
\newblock {\em Computing the continuous discretely}.
\newblock Undergraduate Texts in Mathematics. Springer, New York, 2007.
\newblock Integer-point enumeration in polyhedra.

\bibitem[Cat00]{Cat00}
F.~Catanese.
\newblock Fibred surfaces, varieties isogenous to a product and related moduli
  spaces.
\newblock {\em American Journal of Mathematics}, 122(1):1--44, 2000.

\bibitem[Cat08]{Cat08}
F.~Catanese.
\newblock Differentiable and deformation type of algebraic surfaces, real and
  symplectic structures.
\newblock In {\em Symplectic 4-manifolds and algebraic surfaces}, volume 1938
  of {\em Lecture Notes in Math.}, pages 55--167. Springer, Berlin, 2008.

\bibitem[CCML98]{CCML98}
F.~Catanese, C.~Ciliberto, and M.~Mendes~Lopes.
\newblock On the classification of irregular surfaces of general type with
  nonbirational bicanonical map.
\newblock {\em Trans. Amer. Math. Soc.}, 350(1):275--308, 1998.

\bibitem[CLP12]{CLP12}
F.~{Catanese}, M.~{Loenne}, and F.~{Perroni}.
\newblock {The irreducible components of the moduli space of dihedral covers of
  algebraic curves}.
\newblock ArXiv: 1206.5498, 2012.

\bibitem[Con14]{Conder}
M.~D.~E. Conder.
\newblock Large group actions on surfaces.
\newblock In {\em Riemann and {K}lein surfaces, automorphisms, symmetries and
  moduli spaces}, volume 629 of {\em Contemp. Math.}, pages 77--97. Amer. Math.
  Soc., Providence, RI, 2014.

\bibitem[CP09]{CP09}
G.~Carnovale and F.~Polizzi.
\newblock The classification of surfaces with $p_g = q = 1$ isogenous to a
  product of curves.
\newblock {\em Advances in Geometry}, 9(2):233--256, 2009.

\bibitem[CW34]{ChevWeil}
C.~Chevalley and A.~Weil.
\newblock {\"Uber das Verhalten der Integrale 1. Gattung bei Automorphismen des
  Funktionenk\"orpers}.
\newblock {\em Abhandlungen aus dem Mathematischen Seminar der Universit\"at
  Hamburg}, 10:358--361, 1934.

\bibitem[Deb82]{deb82}
O.~Debarre.
\newblock In\'egalit\'es num\'eriques pour les surfaces de type g\'en\'eral.
\newblock {\em Bull. Soc. Math. France}, 110(3):319--342, 1982.
\newblock With an appendix by A. Beauville.

\bibitem[DP12]{DP10}
T.~Dedieu and F.~Perroni.
\newblock The fundamental group of a quotient of a product of curves.
\newblock {\em J. Group Theory}, 15(3):439--453, 2012.

\bibitem[FK80]{FK80}
H.~M. Farkas and I.~Kra.
\newblock {\em Riemann surfaces}, volume~71 of {\em Graduate Texts in
  Mathematics}.
\newblock Springer-Verlag, New York-Berlin, 1980.

\bibitem[GH94]{G-H}
P.~Griffiths and J.~Harris.
\newblock {\em Principles of algebraic geometry}.
\newblock Wiley Classics Library. John Wiley \& Sons, Inc., New York, 1994.
\newblock Reprint of the 1978 original.

\bibitem[Hat02]{Hatcher}
A.~Hatcher.
\newblock {\em Algebraic Topology}.
\newblock Cambridge university press, 2002.

\bibitem[HP02]{HP02}
C.~D. Hacon and R.~Pardini.
\newblock Surfaces with {$p_g=q=3$}.
\newblock {\em Trans. Amer. Math. Soc.}, 354(7):2631--2638 (electronic), 2002.

\bibitem[Isa76]{Isaacs76}
I.~M. Isaacs.
\newblock {\em Character theory of finite groups}.
\newblock Academic Press [Harcourt Brace Jovanovich, Publishers], New
  York-London, 1976.
\newblock Pure and Applied Mathematics, No. 69.

\bibitem[Lam05]{lamotke}
K.~Lamotke.
\newblock {\em {R}iemannsche {F}l{\"a}chen}.
\newblock Springer-Lehrbuch. Springer {V}erlag, 2005.

\bibitem[Mir95]{mir}
R.~Miranda.
\newblock {\em Algebraic curves and {R}iemann surfaces}, volume~5 of {\em
  Graduate Studies in Mathematics}.
\newblock American Mathematical Society, Providence, RI, 1995.

\bibitem[Miy87]{MiyaokaChern}
Y.~Miyaoka.
\newblock {The Chern classes and Kodaira dimension of a minimal variety}.
\newblock {\em Advanced Studies in Math.}, 10:449--477, 1987.

\bibitem[Pen11]{penegini}
M.~Penegini.
\newblock The classification of isotrivially fibred surfaces with {$p_g=q=2$}.
\newblock {\em Collect. Math.}, 62(3):239--274, 2011.
\newblock With an appendix by S{\"o}nke Rollenske.

\bibitem[Pir02]{Pir02}
{G. P.} Pirola.
\newblock Surfaces with {$p_g=q=3$}.
\newblock {\em Manuscripta Math.}, 108(2):163--170, 2002.

\bibitem[Voi07]{Voisin07}
Claire Voisin.
\newblock {\em Hodge theory and complex algebraic geometry. {I}}, volume~76 of
  {\em Cambridge Studies in Advanced Mathematics}.
\newblock Cambridge University Press, Cambridge, english edition, 2007.
\newblock Translated from the French by Leila Schneps.

\bibitem[Wim95]{wiman}
A.~Wiman.
\newblock {\"U}ber die hyperelliptischen {K}urven und diejenigen vom
  {G}eschlechte $p = 3$, welche eindeutige {T}ransformationen in sich zulassen.
\newblock {\em {B}ihang {K}ongl. {S}venska {V}etenkamps-{A}kademiens
  {H}andlingar}, 21:1--23, 1895.

\bibitem[Zim87]{Zimmermann}
B.~Zimmermann.
\newblock Surfaces and the second homology of a group.
\newblock {\em Monatsh. Math.}, 104(3):247--253, 1987.

\end{thebibliography}

\end{document}